\documentclass[10pt,dvipsnames]{amsart}

\usepackage[utf8]{inputenc}
\usepackage{lmodern}
\usepackage[T1]{fontenc}
\usepackage[english]{babel}

\usepackage{latexsym}
\usepackage{amsmath,amssymb,amsthm}
\usepackage{mathrsfs}
\usepackage{stmaryrd}

\usepackage{enumerate}
\usepackage{enumitem}

\usepackage{rotating}

\usepackage{todonotes}
\usepackage{tikz-cd}
\usetikzlibrary{shapes}
\usepackage{tikz}

\makeatletter
\DeclareRobustCommand{\rvdots}{%
	\vbox{
		\baselineskip4\p@\lineskiplimit\z@
		\kern-\p@
		\hbox{.}\hbox{.}\hbox{.}
}}
\makeatother

\usepackage{xcolor}
\usepackage{color}
\usepackage{hyperref}
\hypersetup{
	colorlinks=true,
	linkcolor=blue,
	citecolor=blue,
	urlcolor=blue
}

\usepackage{ulem}
\usepackage{soul}

\newtheorem{theoremAlph}{Theorem}
\newtheorem{corollaryAlph}[theoremAlph]{Corollary}

\newtheorem{theorem}{Theorem}[section]
\newtheorem{lemma}[theorem]{Lemma}	
\newtheorem{proposition}[theorem]{Proposition}
\newtheorem{corollary}[theorem]{Corollary}
\theoremstyle{definition}
\newtheorem{definition}[theorem]{Definition} 
\newtheorem{remark}[theorem]{Remark}

\theoremstyle{definition} 
\newtheorem*{ack}{Acknowledgements}

\numberwithin{equation}{section}

\newcommand{\C}{\mathbb{C}} 
\newcommand{\R}{\mathbb{R}} 
\newcommand{\Z}{\mathbb{Z}} 
\newcommand{\N}{\mathbb{N}} 

\newcommand{\Ric}{\textup{Ric}}

\newcommand{\ttimes}{\mathrel{\widetilde{\times} }}

\newcommand{\Bigslant}[2]{{\raisebox{0em}{$#1$}\left/\raisebox{-.2em}{$#2$}\right.}}
\newcommand{\bigslant}[2]{{\raisebox{0em}{$#1$}\big/\raisebox{-.2em}{$#2$}}}

\makeatletter
\DeclareRobustCommand*\uell{\mathpalette\@uell\relax}
\newcommand*\@uell[2]{
	\setbox0=\hbox{$#1\ell$}
	\setbox1=\hbox{\rotatebox{10}{$#1\ell$}}
	\dimen0=\wd0 \advance\dimen0 by -\wd1 \divide\dimen0 by 2
	\mathord{\lower 0.1ex \hbox{\kern\dimen0\unhbox1\kern\dimen0}}
}

\begin{document}
	\title[\rmfamily Positive Ricci Curvature on Twisted Suspensions]{\rmfamily Positive Ricci Curvature on Twisted Suspensions}
	\date{\today}
	\author{Philipp Reiser}
	\address{Department of Mathematics, University of Fribourg, Switzerland}
	\email{\href{mailto:philipp.reiser@unifr.ch}{philipp.reiser@unifr.ch}}
	\thanks{The author acknowledges funding by the SNSF-Project 200020E\textunderscore 193062 and the DFG-Priority programme SPP 2026.}
	
	\normalem
	
	\begin{abstract}
		The twisted suspension of a manifold is obtained by surgery along the fibre of a principal circle bundle over the manifold. It generalizes the spinning operation for knots and preserves various topological properties. In this article, we show that Riemannian metrics of positive Ricci curvature can be lifted along twisted suspensions. As an application we show that the maximal symmetry rank of a closed, simply-connected Riemannian manifold of positive Ricci curvature is $(n-2)$ in all dimensions $n\geq 4$. Further applications include simply-connected 6-manifolds whose homology has torsion, (rational) homology spheres in all dimensions at least 4, and manifolds with prescribed third homology.
	\end{abstract}

	\maketitle

	\section{Introduction and Main Results}
	
	Spinning, which was first introduced by Artin \cite{Ar25}, see e.g.\ \cite{Fr05}, is a useful tool in knot theory to construct a higher dimensional knot out of a given knot. A similar construction can be made for any given manifold. Roughly speaking, the \emph{$p$-spinning} of a manifold $M$ is the manifold obtained from $M\times S^p$ by surgery along $\{x\}\times S^p$ for some $x\in M$, see e.g.\ \cite{Su87,Su90}. One can further generalize this construction by replacing the product $M\times S^p$ by a linear sphere bundle with fibre $S^p$ and base $M$.
	
	In the special case $p=1$, the $1$-spinning operation is called \emph{suspension} in \cite{Du22} and denoted by $\Sigma_0 M$ due to its resemblance with the classical suspension operation for topological spaces. Indeed, the operation $\Sigma_0$ increases the dimension by $1$, preserves many of the topological properties of the original manifold and maps spheres to spheres. The corresponding generalized spinning operation was subsequently called the \emph{twisted suspension} in \cite{GR23} and denoted $\Sigma_e M$, where $e\in H^2(M;\Z)$ is the Euler class of the principal $S^1$-bundle over $M$ on which surgery is performed. In both articles \cite{Du22} and \cite{GR23} the (twisted) suspension naturally appears in the study of free circle actions.
	
	When considering these operations in the context of curvature, it was shown by Sha--Yang \cite{SY91} that Riemannian metrics of positive Ricci curvature can be lifted along a generalized $p$-spinning operation for any $p\geq 2$ and along the classical $1$-spinning operation, i.e.\ the suspension $\Sigma_0$. The main result of this article is the following theorem, which completes the Sha--Yang result to all generalized spinning operations.
	
	\begin{theoremAlph}\label{T:Twist_sus_Ric}
		Let $M$ be a closed manifold of dimension $n\geq 3$ that admits a Riemannian metric of positive Ricci curvature and let $e\in H^2(M;\Z)$. Then the twisted suspension $\Sigma_e M$ admits a Riemannian metric of positive Ricci curvature.
	\end{theoremAlph}
	
	In fact, we will show that for any $\ell\in\N_0$ the connected sum $\Sigma_e M\#_\ell(S^2\times S^{n-1})$ admits a Riemannian metric of positive Ricci curvature, see Corollary \ref{C:Twist_sus_conn_sum} below.
	
	Theorem \ref{T:Twist_sus_Ric} cannot be extended to the case $n=2$. Indeed, by Lemma \ref{L:Twist_susp_topology} below, the fundamental group of $\Sigma_0\R P^2$ is isomorphic to $\Z$, which, by the theorem of Bonnet--Myers, implies that this space does not admit a Riemannian metric of positive Ricci curvature.
	
	The main difficulty in the proof of Theorem \ref{T:Twist_sus_Ric}, when compared to the Sha--Yang result, lies in the fact that for a principal $S^1$-bundle it is not clear if there exists a submersion metric of non-negative Ricci curvature with totally geodesic fibres which is a product metric on a given local trivialization, i.e.\ whether it is possible to bring the metric into a \textquotedblleft standard form\textquotedblright\ on the part where the surgery is performed. We will therefore use a doubly warped submersion metric approach where we additionally need to \textquotedblleft untwist\textquotedblright\ the metric locally when performing the surgery.
	
	\begin{remark}
		In \cite{Du22} and \cite{GR23} another suspension operation, denoted $\Sigma_1$ and $\widetilde{\Sigma}_e$, respectively, is introduced, which is obtained by additionally twisting the normal bundle trivialization in the surgery process. This operation can be seen as an analogue of twist spinning for knots introduced by Zeeman \cite{Ze65}. We note that it remains open whether positive Ricci curvature can be lifted along these suspension operations. However, in some cases the suspensions $\Sigma_e$ and $\widetilde{\Sigma}_e$ in fact coincide, see Remark \ref{R:TW_TW_SUSP} below, showing that Theorem \ref{T:Twist_sus_Ric} extends to $\widetilde{\Sigma}_e$ in these cases.
	\end{remark}

	We can use Theorem \ref{T:Twist_sus_Ric} to construct examples of Riemannian manifolds of positive Ricci curvature with an isometric torus action of low cohomogeneity. Recall that the \emph{symmetry rank} of a Riemannian manifold, which was introduced by Grove and Searle \cite{GS94}, is the rank of its isometry group. It was asked in \cite{CG20} and \cite{Mo22} what the maximal symmetry rank for a closed, simply-connected $n$-dimensional Riemannian manifold of positive Ricci curvature is. By \cite[Corollary D]{CG20}, see also \cite[p.\ 23]{Mo22} and \cite[p.\ 3796]{Mo22a}, it is given by $(n-2)$ in dimensions $n=4,5,6$, and in dimension $n\geq 7$ it lies between $(n-4)$ and $(n-2)$. We can now use the twisted suspension to show that it is given by $(n-2)$ in all dimensions $n\geq 4$.

	For that we first note that the twisted suspension admits a circle action induced from the free circle action on the principal $S^1$-bundle used in its definition, and the metric constructed in the proof of Theorem \ref{T:Twist_sus_Ric} is invariant under this action. In particular, the manifold
	\[ \#_\ell (S^2\times S^2)\cong \Sigma_0S^3\#_\ell(S^2\times S^2) \]
	admits a Riemannian metric of positive Ricci curvature that is invariant under a circle action.	As pointed out by Michael Wiemeler, this fact can be used in conjunction with the lifting results of Gilkey--Park--Tuschmann \cite{GPT98} by taking principal torus bundles over this manifold to construct examples of closed, simply-connected Riemannian manifolds of positive Ricci curvature with an isometric torus action of cohomogeneity 3 in any dimension, showing that the maximal symmetry rank for a closed, simply-connected Riemannian manifold of positive Ricci curvature is at least $(n-3)$ in all dimensions at least $4$.
	
	To show that the maximal symmetry rank is given by $(n-2)$ we can now additionally lift an isometric torus action on a given manifold to the twisted suspension, provided this action has fixed points, see Theorem \ref{T:Twist_susp_Ric_equiv} below. We obtain the following result.
	
	\begin{theoremAlph}\label{T:cohomo_2}
		Let $n\geq 4$ and $k\in\N_0$ and assume that $k$ is even when $n=4$. Then there exists a closed simply-connected $n$-dimensional spin manifold $M$ with $b_2(M)=k$ that admits a Riemannian metric of positive Ricci curvature that is invariant under an effective action of a torus of dimension $(n-2)$, and if $n\geq5$ and $k\geq 1$ there also exists a non-spin manifold with these properties. In particular, any closed, 2-connected manifold that admits a torus action of cohomogeneity 2 admits a Riemannian metric of positive Ricci curvature that is invariant under this action.
	\end{theoremAlph}

	We note that the manifolds we construct in the proof of Theorem \ref{T:cohomo_2} are, without any assumptions on symmetries, already known to admit Riemannian metrics of positive Ricci curvature, see Remark \ref{R:Ric>0_conn_sum} below. However, to the best of our knowledge, Theorem \ref{T:cohomo_2} provides the first examples in all dimensions $n\geq 7$ of closed, simply-connected Riemannian $n$-manifolds of positive Ricci curvature with an isometric torus action of rank $(n-2)$.
	
	We now consider further applications of Theorem \ref{T:Twist_sus_Ric}. An immediate consequence is the following. By results of Bérard-Bergery \cite{BB78}, metrics of positive Ricci curvature can be lifted along principal $S^1$-bundles with non-trivial real Euler class (cf.\ Corollary \ref{C:g_Ric_tot-geod} below), see also the work by Gilkey--Park--Tuschmann \cite{GPT98}. It is open, however, if an analogous result holds if the base merely decomposes into a connected sum where each individual summand admits a Riemannian metric of positive Ricci curvature (and we assume that the base has finite fundamental group). It is also open whether connected sums of the total space with other manifolds of positive Ricci curvature are possible. We can now use Theorem \ref{T:Twist_sus_Ric} to cover the following special cases.
	\begin{corollaryAlph}
		\begin{enumerate}
			\item Let $M^{4m}$ be a $4m$-dimensional closed manifold with $m\geq 2$ that admits a Riemannian metric of positive Ricci curvature and let $e\in H^2(M;\Z)$. For $\ell\in\N$ let $P\to M\#_\ell (\pm\C P^{2m})$ be the principal $S^1$-bundle whose Euler class restricted to $M$ is $e$ and a generator on each $\pm\C P^{2m}$-summand. Then $P$ admits a Riemannian metric of positive Ricci curvature.
			\item Let $M^n$ be a closed, simply-connected manifold of dimension $n\geq 5$ that admits a Riemannian metric of positive Ricci curvature and let $P\to M$ be a principal $S^1$-bundle with primitive Euler class (or, equivalently, $P$ is simply-connected). Let $N$ be the product $S^2\times S^{n-1}$ if $M$ is non-spin and the non-trivial linear $S^{n-1}$-bundle over $S^2$ if $M$ is spin. Then, for any $\ell\in\N_0$, the manifold $P\#_\ell N$ admits a Riemannian metric of positive Ricci curvature.
		\end{enumerate}\label{C:S1-bundle}
	\end{corollaryAlph}
	To prove Corollary \ref{C:S1-bundle} we will show that $P$ is diffeomorphic to $\Sigma_e M\#_{\ell-1}(S^2\times S^{4m-1})$ in (1) and that $\Sigma_e M$ is diffeomorphic to $P\# N$ in (2). We note that in (1) the metric constructed on $P$ is not invariant under the (free) $S^1$-action.
	
	Next we focus on dimension 6. A large class of examples of closed, simply-connected $6$-manifolds with a Riemannian metric of positive Ricci curvature was constructed in \cite{Re22a,Re23} and a list of all known examples is given in \cite[Section 5.1]{Re23}. However, none of these examples has torsion in its homology. We can now use Theorem \ref{T:Twist_sus_Ric} in combination with the work of Boyer--Galicki \cite{BG02,BG06a,BG08} and Kollár \cite{Ko05,Ko09} on positive Sasakian structures in dimension 5 to construct examples where infinitely many finite groups can be realized as the torsion group of the second homology.
	
	\begin{theoremAlph}\label{T:Hom_Sphere}
		\begin{enumerate}
			\item For every $k\in\N$ that is not a multiple of $30$ there exists a closed, simply-connected rational homology $6$-sphere $\Sigma^6$ with $H_2(\Sigma;\Z)\cong H_3(\Sigma;\Z)\cong \Z/k\oplus\Z/k$ and $H_4(\Sigma;\Z)=0$ such that for any $\ell\in\N_0$ the manifold
			\[ \Sigma\#_\ell(S^2\times S^4) \]
			admits a Riemannian metric of positive Ricci curvature.
			\item For any $k\in\N$ there exists a closed, simply-connected 6-manifold $M_k$, which we can assume to be either spin or non-spin, with $\textrm{Tors}(H_2(M_k;\Z))\cong\Z/k\oplus\Z/k$ such that for any $\ell\in\N_0$ the manifold
			\[ M_k\#_{\ell}(S^2\times S^4) \]
			admits a Riemannian metric of positive Ricci curvature.
		\end{enumerate}
	\end{theoremAlph}
	
	When considering non-simply-connected rational homology spheres, lens spaces provide examples of rational homology spheres with cyclic fundamental group in odd dimensions. By taking twisted suspensions of lens spaces, we can construct such examples in even dimensions as well. More generally, we can take twisted suspensions of spherical space forms such as the Poincaré homology sphere. In the latter case, we even obtain an integral homology sphere.
	
	\begin{theoremAlph}
		\begin{enumerate}
			\item For any integer $k$ and any even dimension $n\geq 4$ there exists a rational homology sphere $M^n_k$ with $\pi_1(M_k)\cong\Z/k$ and $H_i(M_k;\Z)=0$ for all $2\leq i\leq n-1$ such that for any $\ell\in \N_0$ the manifold
			\[ M_k\#_{\ell}(S^2\times S^{n-2}) \]
			admits a Riemannian metric of positive Ricci curvature.
			\item For any $n\geq 4$ there exists a homology $n$-sphere $\Sigma^n$ whose fundamental group is the binary icosahedral group such that for any $\ell\in\N_0$ the manifold
			\[ \Sigma\#_\ell(S^2\times S^{n-2}) \]
			admits a Riemannian metric of positive Ricci curvature.
		\end{enumerate}\label{T:Space_forms}
	\end{theoremAlph}
	
	As final application we construct examples of odd-dimensional simply-connected manifolds of positive Ricci curvature with prescribed third homology group. We note that it was shown in \cite{CW17} that any $2$-connected $7$-manifold admits a Riemannian metric of positive Ricci curvature and in particular any finitely generated abelian group appears as the third homology group of such a manifold. By taking products or twisted suspensions one can extend this result to higher dimensions and to homology groups of higher degrees as well. Here we give an alternative construction with a rather simple rational cohomology ring.
		
	\begin{theoremAlph}\label{T:prescr_G}
		Let
		\[ G\cong \bigslant{\Z}{k_1\Z}\oplus\dots\oplus\bigslant{\Z}{k_{\ell_1}\Z}\oplus\Z^{\ell_2} \]
		be an arbitrary finitely generated abelian group and set $\ell=\ell_1+\ell_2$. Then, for any $m\geq 3$, there exists a closed, simply-connected manifold $M^{2m+1}_G$ of positive Ricci curvature with the rational cohomology ring of $\#_{\ell} (S^2\times S^{2m-1})\#_{\ell_2}(S^3\times S^{2m-2})$ and such that $H_3(M_G)\cong G$.
	\end{theoremAlph}
	
	For homology groups of smaller degree, any finite group is a subgroup of $\mathrm{SU}(n)$ for some $n$ and can therefore be realized as the fundamental group of a closed manifold of positive Ricci curvature. Alternatively, if one only considers the first homology group, appropriate products of lens spaces can produce any given finite abelian group as fundamental group. Note that, by the theorem of Bonnet--Myers, infinite fundamental groups cannot occur. If we fix the dimension, there are also restrictions for finite groups, see \cite{CW06}. For the second homology, to the best of our knowledge, the corresponding problem is open.

	This article is laid out as follows. In Section \ref{S:PREL} we recall basic facts on principal $S^1$-bundles and introduce the twisted suspension. In Section \ref{S:Ric>0} we construct metrics of positive Ricci curvature on twisted suspensions to prove Theorem \ref{T:Twist_sus_Ric}. Finally, in Section \ref{S:Appl} we consider applications and prove Theorems \ref{T:cohomo_2}--\ref{T:prescr_G}. In Appendix \ref{A:cohom_2} we recall basic facts on cohomogeneity-two torus actions and prove a proposition we need in the proof of Theorem \ref{T:cohomo_2}.
	
	\begin{ack}
		The author would like to thank Fernando Galaz-García and Philippe Kupper for helpful discussions, and them as well as Sam Hagh Shenas Noshari, Michael Wiemeler and the anonymous referee for helpful comments on an earlier version of this article. The author would also like to thank the Department of Mathematical Sciences of Durham University for its hospitality during a stay where parts of this work were carried out.
	\end{ack}
	
	\section{Preliminaries}\label{S:PREL}
			
	\subsection{Principal $S^1$-bundles}\label{SS:S1-bundles}
	
	In this section we review basic facts on the topology and the geometry of principal $S^1$-bundles. See for example \cite{Ba14}, \cite{Mo01} for further details. All manifolds, bundles and maps between manifolds are, unless stated otherwise, considered in the smooth category.
	
	Recall that for a manifold $M$ isomorphism classes of principal $S^1$-bundles $P\xrightarrow{\pi}M$ over $M$ are in bijection with the second integral cohomology group $H^2(M;\Z)$, a bijection is given by the Euler class $e(\pi)\in H^2(M;\Z)$.
	
	Now given a principal $S^1$-bundle $P\xrightarrow{\pi}M$ we define the \emph{vertical distribution} $\mathcal{V}\subseteq TP$ by $\mathcal{V}=\ker(d\pi)$, i.e.\ $\mathcal{V}$ is the line bundle tangent to each fibre. A canonical section of $\mathcal{V}$ is given by the action field $\partial_t$ defined  by $\partial_t=\frac{d}{dt}|_{t=0}x\cdot e^{it}$ for each $x\in P$.
	
	A \emph{connection form} is a differential 1-form $A$ on $P$ such that $A(\partial_t)=1$ and $A$ is invariant under the $S^1$-action. From a connection form $A$ we obtain the \emph{horizontal distribution} $\mathcal{H}\subseteq TP$ via $\mathcal{H}=\ker(A)$.
	
	The \emph{curvature form} $\mathcal{F}$ associated to a connection form $A$ is the 2-form $\mathcal{F}=dA$. Since $\mathcal{F}$ is invariant under the $S^1$-action and trivial on $\mathcal{V}$, it is the pull-back of a form on $M$, which we will also denote by $\mathcal{F}$.
	\begin{proposition}[{\cite[Satz 3.23]{Ba14}}]\label{P:CURV_FORM}
		The cohomology class of $-\frac{1}{2\pi}\mathcal{F}$ in $H^2_{dR}(M)$ equals the image of $e(\pi)$ in $H^2(M;\R)$. Moreover, for every representative $\omega$ of this class there exists a connection form $A$ whose curvature form satisfies $-\frac{1}{2\pi}\mathcal{F}=\omega$.
	\end{proposition}

	Given a Riemannian metric $\check{g}$ on $M$, a connection form $A$ and a smooth function $\phi\colon M\to \R$, we define the following $S^1$-invariant metric $g$ on $P$:
	\begin{equation}\label{EQ:g}
		g=\pi^*\check{g}+e^{2\phi}A^2.
	\end{equation}
	By construction $\pi\colon(P,g)\to(M,\check{g})$ is a Riemannian submersion with horizontal distribution given by $\mathcal{H}$. We will be interested in the Ricci curvatures of this metric.
	\begin{lemma}[{\cite[Lemma 1.3]{GPT98}}]\label{L:g_Ric}
		Let $P\xrightarrow{\pi}M$ be a principal $S^1$-bundle and let $g$ be the metric on $P$ defined in \eqref{EQ:g}. Then, for horizontal vectors $X,Y\in\mathcal{H}$, the Ricci curvatures of the metric $g$ are given as follows, where we set $T=e^{-\phi}\partial_t$.
		\begin{align*}
			\Ric(T,T)&= \Delta \phi-\|d\phi\|^2+\frac{e^{2\phi}}{4}\|\mathcal{F}\|^2,\\
			\Ric(T,X)&=\frac{e^\phi}{2}\left(-\delta \mathcal{F}(X)+3\mathcal{F}(X,\nabla\phi) \right),\\
			\Ric(X,Y)&=\Ric_{\check{g}}(X,Y)-\frac{e^{2\phi}}{2}\sum_i\mathcal{F}(X,e_i)\mathcal{F}(Y,e_i)-\textup{Hess}_\phi(X,Y)-X(\phi)Y(\phi).
		\end{align*}
	\end{lemma}
	Here $(e_i)$ is a horizontal orthonormal basis and we identified $X$ and $Y$ with their images under $\pi_*$. We use the convention $\Delta\phi=-\textup{tr}(\textup{Hess}_\phi)$.
	
	\begin{corollary}\label{C:g_Ric_tot-geod}
		In the situation of Lemma \ref{L:g_Ric} suppose that $\phi$ is constant. Then the Ricci curvatures of $g$ are given as follows:
		\begin{align*}
			\Ric(T,T)&= \frac{e^{2\phi}}{4}\|\mathcal{F}\|^2,\\
			\Ric(T,X)&=-\frac{e^\phi}{2}\delta \mathcal{F}(X),\\
			\Ric(X,Y)&=\Ric_{\check{g}}(X,Y)-\frac{e^{2\phi}}{2}\sum_i\mathcal{F}(X,e_i)\mathcal{F}(Y,e_i).
		\end{align*}
	\end{corollary}
	Corollary \ref{C:g_Ric_tot-geod} shows that, by choosing $A$ so that $\mathcal{F}$ is harmonic, the metric $g$ has non-negative Ricci curvature for $\phi$ sufficiently small, provided $\check{g}$ has positive Ricci curvature and $M$ is compact. To obtain strictly positive Ricci curvature one can use the fact that if $P$ has finite fundamental group, the form $\mathcal{F}$ does not vanish identically, showing that there is a point where all Ricci curvature of $g$ are positive. The deformation results of \cite{Eh76} allow then to deform the metric to have positive Ricci curvature everywhere. This is basically the approach taken in \cite{BB78}. Alternatively, instead of perturbing the metric, one can perturb the function $\phi$ as in \cite{GPT98} to obtain strictly positive Ricci curvature. The latter approach has the advantage that the resulting metric is again of the form \eqref{EQ:g}.

	\subsection{Twisted suspensions}\label{SS:Twisted_sus}
	
	In this section we consider the twisted suspension as introduced in \cite{GR23}, which generalizes the suspension operation introduced by Duan \cite{Du22} and is based on the spinning operation for knots due to Artin \cite{Ar25}.
	
	For a connected manifold $M^n$ and a class $e\in H^2(M;\Z)$ we consider the unique principal $S^1$-bundle $P\xrightarrow{\pi}M$ with Euler class $e$. Let $D^n\hookrightarrow M$ be an embedding, which we assume to be orientation preserving if $M$ is orientable, and let $\varphi_\pi\colon D^n\times S^1\hookrightarrow P$ be a local trivialization covering this embedding.
	
	\begin{definition}\label{D:twist_susp}
		The \emph{suspension of $M$ twisted by $e$}, denoted $\Sigma_e M$, is defined by
		\[ \Sigma_e M=P\setminus \varphi_\pi(D^n\times S^1)^\circ\cup_{\mathrm{id}_{S^{n-1}\times S^1}}(S^{n-1}\times D^2).  \]
	\end{definition}
	For $e=0$ the twisted suspension $\Sigma_0 M$ is the suspension introduced by Duan \cite{Du22}. Note that in \cite{GR23} also the twisted suspension $\widetilde{\Sigma}_e M$, which is obtained by gluing along a non-trivial diffeomorphism in Definition \ref{D:twist_susp}, is introduced. We will not consider this suspension operation in this article and refer to Remark \ref{R:TW_TW_SUSP} below, where we discuss in which cases the two suspension operations coincide.
	
	Some topological properties of the space $\Sigma_e M$ are given as follows:
	\begin{lemma}[{\cite[Lemma 5.2]{GR23}}]\label{L:Twist_susp_topology}
		Let $M^n$ be a connected manifold with $n\geq 2$ and let $e\in H^2(M;\Z)$. Then the fundamental group and (co-)homology of $\Sigma_e M$ are given as follows:
		\begin{enumerate}
			\item $\pi_1(\Sigma_e M)\cong \pi_1(M)$ if $n> 2$ and $\pi_1(\Sigma_e M)\cong \pi_1(M\setminus D^2)$ if $n=2$,
			\item $H^2(\Sigma_e M)\cong H^2(M)$ if $M$ is simply-connected and similarly for homology, where we can choose coefficients in any commutative ring, and $\Sigma_e M$ is spin if and only if $w_2(M)\equiv e\mod 2$,
			\item The inclusion $P\setminus\varphi_\pi(D^n\times S^1)^\circ\hookrightarrow \Sigma_e M$ induces isomorphisms in\linebreak (co-)homology in all degrees $3\leq i\leq n$ with coefficients in any commutative ring.
		\end{enumerate}
	\end{lemma}
	For $e=0$ we can give a more explicit description of the (co-)homology groups using Lemma \ref{L:Twist_susp_topology}:
	\begin{equation}\label{EQ:Twist_sus_hom}
		H_i(\Sigma_0 M)\cong \begin{cases}
			H_i(M)\oplus H_{i-1}(M),\quad & i\neq 0,1,n,n+1\\
			H_i(M),\quad & i=0,1,\\
			H_{i-1}(M),\quad & i=n,n+1,
		\end{cases}
	\end{equation}
	(with coefficients in any commutative ring) and similarly for cohomology, cf.\ also \cite[Proposition 3.9]{Du22}. In particular, if $M$ is a (rational) homology sphere, the twisted suspension $\Sigma_0 M$ is also a (rational) homology sphere.

	An alternative description of the space $\Sigma_e M$ can be given in terms of plumbing. We refer to \cite{CW17}, \cite{Re23} and the references given therein for an introduction to plumbing. As in \cite{Re22a} we use the convention that the result of plumbing according to a non-connected graph is the boundary connected sum of the manifolds resulting from plumbing according to each connected component.	
	
	Denote by $\overline{\pi}\colon\overline{P}\to M$ the disc bundle corresponding to $\pi$ and by $\underline{D}^n_{S^2}$ the trivial bundle $S^2\times D^n\to S^2$. Then the space $\Sigma_e M$ is the boundary of the manifold obtained by plumbing according to the following graph:
	
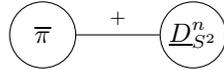
\begin{center}
		\begin{tikzpicture}
			\begin{scope}[every node/.style={circle,draw,minimum height=2.5em}]
				\node[label=center:$\overline{\pi}$] (V1) at (1,0) {\phantom{0}};
				\node[label=center:$\underline{D}^n_{S^2}$] (V2) at (3,0) {\phantom{0}};
			\end{scope}
			\path[-](V1) edge["{$\scriptstyle +$}"] (V2);
		\end{tikzpicture}
	\end{center}
	
	Using this description we can prove the following lemma. Recall that for manifolds $M_1^n,M_2^n$ with $n\geq 4$ we have a natural isomorphism $H^2(M_1\# M_2)\cong H^2(M_1)\oplus H^2(M_2)$.
	\begin{lemma}\label{L:Twist_sus_conn_sum}
		Let $M_1^n,\dots,M_\ell^n$ be connected manifolds with $n\geq 4$ and let $e_i\in H^2(M_i;\Z)$ for all $i=1,\dots,\ell$. Let $e=(e_1,\dots,e_\ell)\in H^2(M_1\#\dots\#M_\ell;\Z)$. Then
		\[ \Sigma_e(M_1\#\dots\# M_\ell)\cong \Sigma_{e_1}M_1\#\dots\#\Sigma_{e_\ell}M_\ell. \]
	\end{lemma}
	\begin{proof}
		Let $P_i\xrightarrow{\pi_i}M_i$, resp.\ $P\xrightarrow{\pi}M_1\#\dots\# M_\ell$, be the principal $S^1$-bundle with Euler class $e_i$, resp.\ $e$. Then the boundary connected sum of the bundles $\pi_1,\dots,\pi_\ell$ is isomorphic to the bundle $\pi$ since both bundles have Euler class $e$. Hence, $P$ is the boundary of the manifold obtained by plumbing as follows, see \cite[Proposition 3.4]{Re22a}.
		\begin{center}
			\begin{tikzpicture}
				\begin{scope}[every node/.style={circle,draw,minimum height=2.5em}]
					\node[label=center:$\overline{\pi}_1$] (V1) at (1,0) {\phantom{0}};
					\node[label=center:$\underline{D}^n_{S^2}$] (V2) at (3,0) {\phantom{0}};
					\node[label=center:$\overline{\pi}_2$] (V3) at (5,0) {\phantom{0}};
					\node[label=center:$\underline{D}^n_{S^2}$] (V4) at (7,0) {\phantom{0}};
					\node[label=center:$\cdots$,draw=none] (V5) at (9,0) {\phantom{0}};
					\node[label=center:$\overline{\pi}_\ell$] (V6) at (11,0) {\phantom{0}};
				\end{scope}
				\path[-](V1) edge["{$\scriptstyle +$}"] (V2);
				\path[-](V2) edge["{$\scriptstyle +$}"] (V3);
				\path[-](V3) edge["{$\scriptstyle +$}"] (V4);
				\path[-](V4) edge["{$\scriptstyle +$}"] (V5);
				\path[-](V5) edge["{$\scriptstyle +$}"] (V6);
			\end{tikzpicture}
		\end{center}
		Hence, the manifold $\Sigma_e(M_1\#\dots\#M_\ell)$ is the result of the following graph.
		\begin{center}
			\begin{tikzpicture}
				\begin{scope}[every node/.style={circle,draw,minimum height=2.5em}]
					\node[label=center:$\overline{\pi}_1$] (V1) at (1,0) {\phantom{0}};
					\node[label=center:$\underline{D}^n_{S^2}$] (V2) at (3,0) {\phantom{0}};
					\node[label=center:$\overline{\pi}_2$] (V3) at (5,0) {\phantom{0}};
					\node[label=center:$\underline{D}^n_{S^2}$] (V4) at (7,0) {\phantom{0}};
					\node[label=center:$\cdots$,draw=none] (V5) at (9,0) {\phantom{0}};
					\node[label=center:$\overline{\pi}_\ell$] (V6) at (11,0) {\phantom{0}};
					\node[label=center:$\underline{D}^n_{S^2}$] (V7) at (13,0) {\phantom{0}};
				\end{scope}
				\path[-](V1) edge["{$\scriptstyle +$}"] (V2);
				\path[-](V2) edge["{$\scriptstyle +$}"] (V3);
				\path[-](V3) edge["{$\scriptstyle +$}"] (V4);
				\path[-](V4) edge["{$\scriptstyle +$}"] (V5);
				\path[-](V5) edge["{$\scriptstyle +$}"] (V6);
				\path[-](V6) edge["{$\scriptstyle +$}"] (V7);
			\end{tikzpicture}
		\end{center}
		By \cite[Proposition 3.2]{Re22a} we can split this graph into the disjoint union of the graphs
		
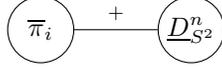
\begin{center}
			\begin{tikzpicture}
				\begin{scope}[every node/.style={circle,draw,minimum height=2.5em}]
					\node[label=center:$\overline{\pi}_i$] (V1) at (1,0) {\phantom{0}};
					\node[label=center:$\underline{D}^n_{S^2}$] (V2) at (3,0) {\phantom{0}};
				\end{scope}
				\path[-](V1) edge["{$\scriptstyle +$}"] (V2);
			\end{tikzpicture}
		\end{center}
		The resulting manifold is therefore the connected sum of the manifolds $\Sigma_{e_i}M_i$.
	\end{proof}
	
	We can repeat the surgery operation in the definition of the twisted suspension by performing multiple surgeries on fibre spheres. The resulting manifold can again be expressed using the twisted suspension.
	\begin{lemma}\label{L:Twist_susp_multiple}
		Let $M^n$ be a connected manifold, let $e\in H^2(M;\Z)$ and let $P\xrightarrow{\pi}M$ be the principal $S^1$-bundle with Euler class $e$. Let $\varphi_{1},\dots,\varphi_\ell\colon D^n\times S^1\hookrightarrow P$ be local trivializations as in Definition \ref{D:twist_susp} with pairwise disjoint image. Then the space
		\[ P'=\left(P\setminus\bigcup_i\varphi_i(D^n\times S^1)^\circ\right) \cup_{\textup{id}_{\sqcup_i (S^{n-1}\times S^1) } }\left(\bigsqcup_i(S^{n-1}\times D^2)\right) \]
		is diffeomorphic to the connected sum
		\[ P'\cong \Sigma_e M\#_{\ell-1}(S^2\times S^{n-1}). \]
	\end{lemma}
	\begin{proof}
		The space $P'$ is the result of plumbing according to the following graph.
		\begin{center}
			\begin{tikzpicture}
				\begin{scope}[every node/.style={circle,draw,minimum height=2.5em}]
					\node (U) at (1,0) {$\overline{\pi}$};
					\node[label=center:$\underline{D}^n_{S^2}$] (G1) at (3,1) {\phantom{0}};
					\node[label=center:$\underline{D}^n_{S^2}$] (Gn) at (3,-1) {\phantom{0}};
					\node[draw=none] (dots) at (2,0) {$\rvdots_\ell$};
				\end{scope}
				\path[-](U) edge["{$\scriptstyle +$}"] (G1);
				\path[-](Gn) edge["{$\scriptstyle +$}"] (U);
			\end{tikzpicture}
		\end{center}
		By \cite[Proposition 3.2]{Re22a} we can remove all except one edge of this graph without changing the diffeomorphism type. In particular, $P'$ is diffeomorphic to the connected sum of $\Sigma_e M$ and $(\ell-1)$ copies of $S^2\times S^{n-1}$.		
	\end{proof}
	
	\begin{remark}\label{R:TW_TW_SUSP}
		In \cite{GR23} another suspension operation, denoted $\widetilde{\Sigma}_e$, is introduced, which is obtained by replacing the identity map on $S^{n-1}\times S^1$ in Definition \ref{D:twist_susp} by a map $\tilde{\alpha}\colon S^{n-1}\times S^1\to S^{n-1}\times S^1$ defined by $\tilde{\alpha}(x,y)=(\alpha_y x,y)$, where $\alpha\colon S^1\to\mathrm{SO}(n)$ represents a generator of $\pi_1(\mathrm{SO}(n))$. In general, the operations $\Sigma_e$ and $\widetilde{\Sigma}_e$ do not coincide. For example, if $M$ is a simply-connected non-spin manifold and $e\in H^2(M;\Z)$ satisfies $e\equiv w_2(M)\mod 2$, then $\Sigma_e M$ is spin, while $\widetilde{\Sigma}_e M$ is non-spin, see \cite[Lemma 5.2]{GR23}.
		
		However, if the universal cover of the total space $P$ of the principal $S^1$-bundle $P\xrightarrow{\pi} M$ is non-spin, then the spaces $\Sigma_e M$ and $\widetilde{\Sigma}_eM$ are in fact diffeomorphic. Indeed, in this case, by \cite[Proposition 3.6]{GR23}, any two embeddings $S^1\times D^n\hookrightarrow P$ that induce the same map on fundamental groups are isotopic. In particular, the embeddings $\varphi_\pi$ and $\varphi_\pi\circ\tilde{\alpha}$ are isotopic, which shows that the spaces $\Sigma_e M$ and $\widetilde{\Sigma}_eM$ are diffeomorphic.
		
		To characterise when the universal cover $\widetilde{P}$ of $P$ is non-spin, consider the pull-back of $P\xrightarrow{\pi}M$ along the universal cover $\widetilde{M}\xrightarrow{\pi_M}M$ of $M$ and denote the resulting space by $\hat{P}$. Then $\hat{P}\to\widetilde{M}$ is a principal $S^1$-bundle with Euler class $\pi_M^*e$ and $\hat{P}\to P$ is a covering. If $\hat{P}\to\widetilde{M}$ is the trivial bundle, then the universal cover is given by $\R\times \widetilde{M}$, which is spin if and only if $\widetilde{M}$ is spin. If $\hat{P}\to\widetilde{M}$ is non-trivial, denote by $d\in\N$ the divisibility of $\pi_M^*e\in H^2(\widetilde{M};\Z)$ (and note that this group is torsion-free since $\widetilde{M}$ is simply-connected). Then the total space $\widetilde{P}$ of the principal $S^1$-bundle over $\widetilde{M}$ with Euler class $\frac{1}{d}\pi^*e$ is simply-connected (see e.g.\ \cite[Lemma 2.3]{GR23}) and covers $\hat{P}$ and is therefore the universal cover of $\hat{P}$ and $P$. By \cite[Corollary 2.6]{GR23}, $\widetilde{P}$ is non-spin if and only if $\widetilde{M}$ is non-spin and $\frac{1}{d}\pi_M^*e\not\equiv w_2(\widetilde{M})\mod 2$.
		
		If one restricts to the suspension operations $\Sigma_0$ and $\widetilde{\Sigma}_0$, then it was shown in \cite{Du22} that for a manifold $M^n$ the spaces $\Sigma_0 M$ and $\widetilde{\Sigma}_0 M$ are diffeomorphic whenever $M$ is \emph{$\Sigma$-stable}, that is, whenever the map $\mathrm{Diff}(M,x_0)\to\mathrm{GL}(n)$, which assigns to a diffeomorphism of $M$ fixing $x_0\in M$ its differential on $T_{x_0}M$, induces a surjective map on fundamental groups. Examples of $\Sigma$-stable manifolds include spheres (see \cite[Proposition 3.2]{Du22}) and products $M_1\times M_2$ where at least one factor is $\Sigma$-stable (see \cite[Corollary 3.1]{Du22}).
	\end{remark}

	\section{Positive Ricci curvature on Twisted Suspensions}\label{S:Ric>0}
	
	To prove Theorem \ref{T:Twist_sus_Ric} we will show the following more general result.
	
	\begin{theorem}\label{T:Twist_sus_Ric_multiple}
		Let $M^n$ be a closed manifold of dimension $n\geq 3$ that admits a Riemannian metric of positive Ricci curvature and let $e\in H^2(M;\Z)$. Let $P\xrightarrow{\pi}M$ be the principal $S^1$-bundle with Euler class $e$ and let $\varphi_1,\dots,\varphi_\ell\colon D^n\times S^1\hookrightarrow P$ be local trivializations covering embeddings $D^n\hookrightarrow M$ with pairwise disjoint image. Then the space
		\begin{equation}\label{EQ:gluing}
			P'=\left(P\setminus\bigcup_i\varphi_i(D^n\times S^1)^\circ\right) \cup_{\textup{id}_{\sqcup_i (S^{n-1}\times S^1) } }\left(\bigsqcup_i(S^{n-1}\times D^2)\right)
		\end{equation}
		admits a Riemannian metric of positive Ricci curvature.
	\end{theorem}
	Using Lemma \ref{L:Twist_susp_multiple} we obtain the following corollary, which implies Theorem \ref{T:Twist_sus_Ric} for $\ell=1$.
	\begin{corollary}\label{C:Twist_sus_conn_sum}
		Let $M^n$ be a closed manifold of dimension $n\geq 3$ that admits a Riemannian metric of positive Ricci curvature and let $e\in H^2(M;\Z)$. Then for any $\ell\in\N$ the manifold $\Sigma_e M\#_{\ell-1}(S^2\times S^{n-1})$ admits a Riemannian metric of positive Ricci curvature.
	\end{corollary}
	
	To prove Theorem \ref{T:Twist_sus_Ric_multiple} we view $S^{n-1}\times D^2$ as the space obtained from $[0,s_\lambda]\times S^{n-1}\times S^1$ for some $s_\lambda>0$ (which will be specified later) by collapsing each circle $\{0\}\times \{v\}\times S^1$, where $v\in S^{n-1}$. On the product $[0,s_\lambda]\times S^{n-1}\times S^1$, which we view as a principal $S^1$-bundle, we then define a connection metric which we can glue to a connection metric on $P\setminus\bigcup_i\varphi_i(D^n\times S^1)^\circ$. We will now describe the metrics on each part.
	
	\subsection{The metric on $P\setminus\bigcup_i\varphi_i(D^n\times S^1)^\circ$}\label{SS:metric1} Let $x_1,\dots,x_\ell\in M$ be pairwise distinct points. We can now deform the metric of positive Ricci curvature on $M$ into a metric of positive Ricci curvature that has constant sectional curvature $1$ in a neighbourhood of each point $x_i$, see e.g.\ \cite{Wr02}, \cite[Corollary 4.4]{RW23} and Lemma~\ref{L:equiv_deform} below. We denote this metric by $\check{g}$. Thus, we have isometric embeddings $\check{\varphi}_1,\dots,\check{\varphi}_\ell\colon D^n\hookrightarrow M$, where on $D^n$ we consider the induced metric of a geodesic ball of some radius $\varepsilon>0$ in the round sphere of radius $1$. By the disc theorem of Palais \cite{Pa59} we can assume that these embeddings are the original embeddings $D^n\hookrightarrow M$ and we obtain local trivializations $\varphi_1,\dots,\varphi_\ell$ covering these embeddings.
	
	By the Hodge Theorem, see e.g.\ \cite[Theorem 4.16]{Mo01}, there exists a unique harmonic 2-form $\omega$ on $M$ that represents the class $e$. By Proposition \ref{P:CURV_FORM} there exists a connection form $A$ on $P$ with curvature form $\mathcal{F}=-2\pi\omega$. In particular, $\mathcal{F}$ is harmonic, so $\delta\mathcal{F}=0$. For a constant $\phi>0$ let $g_\phi$ be the metric \eqref{EQ:g}. It follows from Corollary \ref{C:g_Ric_tot-geod} that $g_\phi$ has non-negative Ricci curvature for $\phi$ sufficiently small.
	
	\subsection{The metric on $[0,s_\lambda]\times S^{n-1}\times S^1$} On $[0,s_\lambda]\times S^{n-1}\times S^1$ we define a doubly warped submersion metric as follows. Let $s_\lambda>0$ and let $f,h\colon [0,s_\lambda]\to[0,\infty)$ be smooth functions so that $f$ is strictly positive and $h$ only vanishes at $s=0$. On $[0,s_\lambda]\times S^{n-1}$ we then define the metric
	\[ \check{g}_f=ds^2+f(s)^2ds_{n-1}^2. \]
	On the principal $S^1$-bundle $[0,s_\lambda]\times S^{n-1}\times S^1$ we choose an arbitrary connection form $A$ that is given by $A=dt$ in a neighbourhood of $s=0$. We then define the metric \eqref{EQ:g} given by
	\[g_{f,h}=\check{g}_f+e^{2\phi}A^2, \]
	where $\phi=\ln(h)$. Since $A=dt$ in a neighbourhood of $s=0$, the metric $g_{f,h}$ is a doubly warped product metric in this neighbourhood. By \cite[Proposition 1.4.7]{Pe16} we obtain a smooth metric on the quotient $S^{n-1}\times D^2$ if and only if $f$ is even at $s=0$ and $h$ is odd at $s=0$ with $h'(0)=1$.	
	
	We will use the following warping functions.
	\begin{lemma}\label{L:warping_fcts}
		For every $\lambda\in(0,1)$ there exists $s_\lambda>0$ and smooth functions $f,h\colon [0,s_\lambda]\to[0,\infty)$ satisfying the differential inequalities
		\begin{align}
			& \label{INEQ:1}-(n-1)\frac{f''}{f}-\frac{h''}{h}>0,\\
			& \label{INEQ:2}-\frac{f''}{f}+(n-2)\frac{1-{f'}^2}{f^2}-\frac{f'h'}{fh}>0,\\
			& \label{INEQ:3}-\frac{h''}{h}-(n-1)\frac{f'h'}{fh}> 0,\\
			& \label{INEQ:4}f>0,\quad h|_{(0,s_\lambda]}>0,
		\end{align}
		and the boundary conditions
		\begin{enumerate}
			\item $f$ is even at $s=0$ and $h$ is odd at $s=0$ with $h'(0)=1$,
			\item There exist $N>0$, $s'\in\R$ such that $f(s)=N\sin\left(\frac{s-s'}{N}\right)$ in a neighbourhood of $s=s_\lambda$ and $f'(s_\lambda)=\lambda$, and
			\item All derivatives of $h$ vanish at $s=s_\lambda$.
		\end{enumerate}
	\end{lemma}
	Note that the boundary conditions in particular imply that $f'(0)=h(0)=0$.
	\begin{proof}
		The construction in \cite{SY91} already provides functions $f$ and $h$ that satisfy all required conditions, except that \eqref{INEQ:1} and \eqref{INEQ:2} merely hold as non-strict inequalities. Since it will be crucial that these hold strictly in our case, we modify this construction as follows.
		
		Let $\lambda_0\in(\lambda,1)$ and choose $\alpha\in(n-2,\frac{n-2}{\lambda_0^2})$. Now define $f\colon [0,\infty)\to (0,\infty)$ as the unique smooth function satisfying
		\begin{align*}
			f(0)&=1,\\
			f'(0)&=0,\\
			f''&= \frac{\alpha\lambda_0^2}{2}f^{-\alpha-1}.
		\end{align*}
		From the definition it follows that $f''>0$, and since $f'(0)=0$, the derivative $f'$ is positive and monotone increasing on $(0,\infty)$. In particular, $f(s)$ converges to $\infty$ as $s\to\infty$.
		
		By integrating both sides of the equation $f'' f'=\frac{\alpha\lambda_0^2}{2}f^{-\alpha-1}f'$ we obtain
		\[ f'^2=\lambda_0^2\left(1- f^{-\alpha}\right). \]
		Hence, since $f(s)\to \infty$ as $s\to\infty$, it follows that $f'$ converges to $\lambda_0>\lambda$. In particular, there exists $s_\lambda>0$ with $f'(s_\lambda)=\lambda$.		
		
		Next, we define $h\colon [0,\infty)\to[0,\infty)$ by
		\[ h=\frac{2}{\alpha\lambda_0^2}f'. \]
		We then have
		\begin{align*}
			h'&= f^{-\alpha-1},\\
			h''&=-(\alpha+1)f^{-\alpha-2}f'. 
		\end{align*}
		One can show now inductively that $f$ is an even function and $h$ is an odd function at $s=0$. In particular, the required boundary conditions at $s=0$ and inequalities \eqref{INEQ:4} are all satisfied. For the inequalities \eqref{INEQ:1}--\eqref{INEQ:3} we calculate
		\begin{align*}
			\frac{f''}{f}&= \frac{\alpha\lambda_0^2}{2}f^{-\alpha-2},\\
			\frac{h''}{h}&=-\frac{\alpha(\alpha+1)}{2}\lambda_0^2 f^{-\alpha-2},\\
			\frac{1-f'^2}{f^2}&=\frac{1-\lambda_0^2+\lambda_0^2f^{-\alpha}}{f^2},\\
			\frac{f'h'}{fh}&= \frac{\alpha\lambda_0^2}{2}f^{-\alpha-2}. 
		\end{align*}
		From this it follows immediately that inequalities \eqref{INEQ:1} and \eqref{INEQ:3} are satisfied. For \eqref{INEQ:2} we use that $f\geq 1$ and $(n-2)<\alpha<\frac{n-2}{\lambda_0^2}$ as follows:
		\begin{align*}
			-\frac{f''}{f}+(n-2)\frac{1-{f'}^2}{f^2}-\frac{f'h'}{fh}&=f^{-2}\left( (-\alpha\lambda_0^2+(n-2)\lambda_0^2)f^{-\alpha}+(n-2)(1-\lambda_0^2) \right)\\
			&\geq f^{-2}\left( -\alpha\lambda_0^2+(n-2) \right)\\
			&>0.
		\end{align*}
		
		It remains to modify $f$ and $h$ near $s=s_\lambda$ to satisfy the required boundary conditions. Note that we have already achieved $f'(s_\lambda)=\lambda$. Now let $N>0$ and $s'\in\R$ so that the function $s\mapsto N\sin\left(\frac{s-s'}{N}\right)$ extends $f$ as a $C^1$-function at $s=s_\lambda$. Note that, since \eqref{INEQ:1}--\eqref{INEQ:3} depend linearly on the second derivative of $f$, these inequalities are satisfied at $s=s_\lambda$ for the function
		\[ s\mapsto \mu f(s)+(1-\mu)N\sin\left(\frac{s-s'}{N}\right). \]
		for all $\mu\in[0,1]$, and therefore also in a small neighbourhood of $s=s_\lambda$. Thus, by using a suitable cut-off function, we can modify $f$ in a small neighbourhood of $s=s_\lambda$ so that $(2)$ is satisfied and \eqref{INEQ:1}--\eqref{INEQ:3} still hold on $[0,s_\lambda]$, see \cite[Theorem 1.2]{BH22}, cf.\ also \cite{Wr02}.
		
		Finally, we modify $h$ near $s=s_\lambda$ to satisfy $(3)$. For that, let $\varepsilon>0$ and let $\psi\colon\R\to\R$ be a smooth function that is constant $1$ on $(-\infty,s_\lambda-\varepsilon]$ and constant $0$ on $[s_\lambda,\infty)$ with $\psi'\leq 0$ and $\psi>0$ on $(-\infty,s_\lambda]$. We replace $h$ on $[s_\lambda-\varepsilon,s_\lambda]$ by the function $\tilde{h}$ defined by
		\begin{align*}
			\tilde{h}(s_\lambda-\varepsilon)&=h(s_\lambda-\varepsilon),\\
			\tilde{h}'&=\psi\cdot h'.
		\end{align*}
		Then we have $\tilde{h}''=h''\psi+h'\psi'\leq h''\psi$, so
		\[ -\frac{\tilde{h}''}{\tilde{h}'}\geq -\frac{h''}{h'}, \]
		which implies that \eqref{INEQ:3} holds. Moreover, since $\tilde{h}''\leq 0$ and $\tilde{h}'\in[0,h'(s_\lambda-\varepsilon)]$, while $\tilde{h}(s_\lambda)\to h(s_\lambda)$ as $\varepsilon\to 0$ and $f''<0$ in a small neighbourhood of $s=s_\lambda$, also \eqref{INEQ:1} and \eqref{INEQ:2} hold on $[s_\lambda-\varepsilon,s_\lambda]$ for $\varepsilon$ sufficiently small.		
	\end{proof}
	For the metric on $[0,s_\lambda]\times S^{n-1}\times S^1$ the functions $f$ and $h$ will be warping functions for the spheres $S^{n-1}$ and $S^1$, respectively. It will be important to scale down the size of the $S^1$-factor, i.e.\ to scale down the function $h$. While inequalities \eqref{INEQ:1}--\eqref{INEQ:3} are invariant under scaling of $h$, the boundary condition $h'(0)=1$ is not. This, however, can be adjusted easily.
	\begin{lemma}\label{L:warping_fcts_smoothing}
		Let $\lambda\in(0,1)$ and let $s_\lambda$, $f$ and $h$ as in Proposition \ref{L:warping_fcts}. Then for any $r\in(0,1)$ and any $\varepsilon>0$ we can modify the functions $f$ and $r\cdot h$ on $[0,\varepsilon]$ so that on $[\varepsilon',s_\lambda]$ for some $\varepsilon'\in(0,\varepsilon)$ they satisfy the conclusions of Proposition \ref{L:warping_fcts} (and we consider the boundary conditions at $s=\varepsilon'$ instead of $s=0$).
	\end{lemma}
	\begin{proof}
		Since $f'(0)=0$, we can slightly perturb the function $f$ to be constant in a neighbourhood of $s=0$. For given $s_0\in(0,\varepsilon)$ in this neighbourhood we then replace $h$ on $[\varepsilon',s_0]$ with the function
		\[ s\mapsto \frac{R}{r}\sin\left( \frac{s-\varepsilon'}{R} \right), \]
		where $R$ and $\varepsilon'$ are chosen so that $h$ is $C^1$ at $s=s_0$. It is easily verified that all required properties are satisfied. Finally, we smooth $h$ in a small neighbourhood of $s=s_0$ while keeping inequalities \eqref{INEQ:1}--\eqref{INEQ:3} satisfied, cf.\ \cite[Corollary 3.2]{Re23}.
	\end{proof}
	\begin{proposition}\label{P:g_Ric>0}
		For given $\lambda,r\in(0,1)$ let $f,h_r$ be the functions obtained in Lemma \ref{L:warping_fcts_smoothing}, where $\varepsilon$ is chosen so that $A=dt$ on $[0,\varepsilon]$. Then for all $r$ sufficiently small the metric $g_{f,h_r}$ has positive Ricci curvature.
	\end{proposition}
	\begin{proof}
		Let $X,Y$ be local unit vector fields on $S^{n-1}$. We extend $X$ and $Y$ constantly in $s$-direction to obtain local vector fields on $[0,s_\lambda]\times S^{n-1}$. Using the Koszul formula one obtains the following expressions for the Levi--Civita connection $\nabla^{\check{g}_f}$ of $\check{g}_f$.
		\begin{align*}
			\nabla^{\check{g}_f}_{\partial_s}\partial_s&=0,\\
			\nabla^{\check{g}_f}_{\partial_s}X&=\nabla^{\check{g}_f}_X\partial_s=\frac{f'}{f}X,\\
			\nabla^{\check{g}_f}_X Y&=-f f'\langle X,Y\rangle_{S^{n-1}}\partial_s+\nabla^{S^{n-1}}_X Y.
		\end{align*}
		We then obtain from Lemma \ref{L:g_Ric} the following expressions for the Ricci curvatures of $g_{f,h}$:
		\begin{align*}
			\Ric(T,T)=&-\frac{h_r''}{h_r}-(n-1)\frac{f'h_r'}{fh_r}+\frac{h_r^2}{4}\|\mathcal{F}\|_{\check{g}_f}^2,\\
			\Ric\left(X/f,Y/f\right)=&\left( \frac{-f''}{f}+(n-2)\frac{1-f'^2}{f^2}-\frac{f'h_r'}{fh_r} \right)\langle X,Y\rangle_{S^{n-1}}\\
			&-\frac{h_r^2}{2f^4}\sum_i\mathcal{F}(X,e_i)\mathcal{F}(Y,e_i),\\
			\Ric(\partial_s,\partial_s)=& -\frac{h_r''}{h_r}-(n-1)\frac{f''}{f}-\frac{h_r^2}{2f^2}\sum_i\mathcal{F}(\partial_s,e_i)^2,\\
			\Ric\left(T,X/f\right)=&\frac{h_r}{2}\left( -\delta_{\check{g}_f}\mathcal{F}(X)+3\frac{h'_r}{h_r}\mathcal{F}(X,\partial_s ) \right),\\
			\Ric(T,\partial_s)=&-\frac{h_r}{2}\delta_{\check{g}_f}\mathcal{F}(\partial_s),\\
			\Ric(X/f,\partial_s)=&-\frac{h_r^2}{2f^3}\sum_i\mathcal{F}(X,e_i)\mathcal{F}(\partial_s,e_i).
		\end{align*}
		Here $(e_i)$ is an orthonormal basis with respect to $ds_{n-1}^2$.
		
		Since the functions $f,h_r$ satisfy inequalities \eqref{INEQ:1}--\eqref{INEQ:3}, and since $h_r=r\cdot h$ on $[\varepsilon,s_\lambda]$, the Ricci curvatures $\Ric(T,T)$, $\Ric(X/f,X/f)$ and $\Ric(\partial_s,\partial_s)$ are strictly positive for $r$ sufficiently small and converge to a positive function as $r\to 0$. The absolute values of the mixed Ricci curvatures $\Ric(T,X/f)$, $\Ric(T,\partial_s)$ and $\Ric(X/f,\partial_s)$ can be bounded above by any positive constant by possibly choosing $r$ smaller. It follows that $g_{f,h_r}$ has positive Ricci curvature for $r$ sufficiently small.
	\end{proof}
		
		\subsection{Gluing and Deforming} We have now defined metrics of non-negative Ricci curvature on both parts of the gluing \eqref{EQ:gluing}. To finish the proof of Theorem \ref{T:Twist_sus_Ric_multiple} it remains to glue these metrics together to obtain a metric of non-negative Ricci curvature on $P'$ and to deform it into a metric of positive Ricci curvature.
		
		By construction, the metric $\check{g}$ considered in Subsection \ref{SS:metric1} is of the form $ds^2+\sin^2(s)ds_{n-1}^2$ near each $x_i$ where $s$ denotes the distance from $x_i$. Let $s_0>0$ such that for all $x_i$ the metric has this form on the ball of radius $s_0$. We set $\lambda=\cos(s_0)$ and consider $\ell$ copies of the metric $\check{g}_f$ on $[0,s_\lambda]\times S^{n-1}$ with $f$ as obtained in Lemma \ref{L:warping_fcts}. By using a partition of unity we define for each $i$ a connection form $A_i$ on the $i$-th copy of $[0,s_\lambda]\times S^{n-1}\times S^1$ that smoothly extends the connection form given on $P\setminus\bigcup_i\varphi_i(D^n\times S^1)^\circ$ according to the gluing \eqref{EQ:gluing} and is given by $dt$ in a neighbourhood of $s=0$.
		
		For $r\in(0,1)$ we now consider the metric $g_{f,h_r}$ with connection form $A_i$ and with $h_r$ as obtained in Lemma \ref{L:warping_fcts_smoothing} and $f$ modified accordingly. Since $f(s)=N\sin\left(\frac{s-s'}{N}\right)$ in a neighbourhood of $s=s_\lambda$ and $f'(s_\lambda)=\lambda$, the metric $\check{g}_{f,h_r}$ glues smoothly with the metric $g_\phi$ after rescaling by $\frac{1}{N}$ for suitable values of $r$ and $\phi$. By possibly choosing $r$ and $\phi$ smaller, we obtain a smooth metric of non-negative Ricci curvature on $P'$ by Proposition \ref{P:g_Ric>0}.
		
		Finally, since the metric $g_{f,h_r}$ has strictly positive Ricci curvature by Proposition \ref{P:g_Ric>0}, it follows from the deformation results of Ehrlich \cite{Eh76} that this metric can be deformed into a metric of positive Ricci curvature. This concludes the proof of Theorem \ref{T:Twist_sus_Ric_multiple}.
			
	\section{Applications}\label{S:Appl}
	
	\subsection{Isometric torus actions}
	
	In this section we will prove the following result, which can be seen as an equivariant version of Theorem \ref{T:Twist_sus_Ric} and Corollary \ref{C:Twist_sus_conn_sum}.
	
	\begin{theorem}\label{T:Twist_susp_Ric_equiv}
		Let $(M^n,g)$, $n\geq3$, be a closed Riemannian manifold of positive Ricci curvature that admits an effective isometric action of a torus $T^k$. Suppose that this action has at least $\ell$ fixed points and let $e\in H^2(M;\Z)$. Then the manifold
		\[ \Sigma_e M\#_{\ell-1}(S^2\times S^{n-1}) \]
		admits a Riemannian metric of positive Ricci curvature and an effective isometric $T^{k+1}$-action.
	\end{theorem}
	Note that Theorem \ref{T:Twist_susp_Ric_equiv} in particular applies to the trivial action of the trivial group, resulting in a Riemannian metric of positive Ricci curvature which is invariant under a circle action.
	
	To prove Theorem \ref{T:Twist_susp_Ric_equiv}, we need to deform the metric equivariantly around fixed points into a metric of constant sectional curvature $1$.
	\begin{lemma}\label{L:equiv_deform}
		Let $(M^n,g)$ be a Riemannian manifold of positive Ricci curvature and suppose that $G$ is a Lie group that acts isometrically on $(M,g)$. Let $p\in M$ be a fixed point of this action. Then there exists a metric $g'$ of positive Ricci curvature on $M$ with the following properties:
		\begin{enumerate}
			\item There exists $\varepsilon>0$ so that on $D_\varepsilon(p)\subseteq M$, which we identify with $D_\varepsilon(0)\subseteq T_p M$ via $\exp$, the action is linear and the metric $g'$ coincides with the induced metric of a geodesic ball of radius $\varepsilon$ in the sphere of radius 1,
			\item The metric $g'$ coincides with $g$ outside an arbitrarily small neighbourhood of $p$,
			\item The metric $g'$ is invariant under the action of $G$.
		\end{enumerate}
	\end{lemma}
	\begin{proof}
		We consider the induced linear action of $G$ on the tangent space $T_pM$ via the differential. Since the action of $G$ on $M$ is isometric, the action of $G$ on $T_p M$ preserves the inner product $g_p$. Hence, if we equip $D_{\frac{\pi}{2}}(0)\subseteq T_p M$, i.e.\ the ball of radius $\frac{\pi}{2}$ with respect to $g_p$, with the Riemannian metric $ds_n^2$, i.e.\ round metric of radius $1$, the metric $ds_n^2$ is invariant under the action of $G$.
		
		Now let $\varepsilon>0$ so that the exponential map $\exp\colon T_pM\to M$ is a diffeomorphism on $D_\varepsilon(0)$ and consider the metric
		\[ g_\varepsilon=\left(\exp|_{D_\varepsilon(0)}^{-1}\right)^*ds_n^2 \]
		on $D_\varepsilon(p)\subseteq M$. Since the exponential map is equivariant with respect to the $G$-action, the metric $g_\varepsilon$ is $G$-invariant. Furthermore, since $g_\varepsilon$ and $g$ induce the same normal coordinate system around $p$, their $1$-jets coincide at $p$. Hence, by \cite[Theorem 1.10]{Wr02}, there exists a cutoff function $\psi\colon [0,\infty)\to[0,1]$ with $\psi|_{[0,\varepsilon']}\equiv1$ for some $\varepsilon'\in(0,\varepsilon)$ and $\psi|_{[\varepsilon,\infty)}\equiv 0$, such that the metric
		\[ g'=\psi(d_p)g_\varepsilon+(1-\psi(d_p))g \]
		has positive Ricci curvature, where $d_p\colon M\to \R$ denotes the distance function to $p$. Since $d_p$ is invariant under the $G$-action, the metric $g'$ is invariant under the $G$-action as well and, by definition, coincides with $g_\varepsilon$ in a neighbourhood of $p$ and with $g$ outside $D_\varepsilon(p)$.
	\end{proof}
	
	\begin{proof}[Proof of Theorem \ref{T:Twist_susp_Ric_equiv}]
		We first apply Lemma \ref{L:equiv_deform} to $\ell$ distinct fixed points, so we obtain isometric embeddings $\check{\varphi}_1,\dots,\check{\varphi}_\ell\colon D^n\hookrightarrow M$, where on $D^n$ we consider the induced metric of a geodesic ball of some radius $\varepsilon>0$ in the round sphere of radius $1$, and the action on each of these discs is linear. As in the proof of Theorem \ref{T:Twist_sus_Ric_multiple}, we denote by $\varphi_1,\dots,\varphi_\ell\colon D^n\times S^1\hookrightarrow P$ local trivializations covering the embeddings $\check{\varphi}_1,\dots,\check{\varphi}_\ell$.
		
		The action on each local trivialization $D^n\times S^1$ is now of the form
		\[h(x,y)\mapsto (h\cdot x,\phi(h)y),\]
		where $h\in T^k$, $(x,y)\in D^n\times S^1$, $h\cdot x$ denotes the action of $T^k$ on $M$, and $\phi\colon D^n\to S^1$ is a smooth map. We can modify $\phi$ to be constant $1$ in a neighbourhood of $0\in D^n$. By restricting the local trivializations $\varphi_i$ to sufficiently small discs, we can then assume that the resulting action is trivial on the $S^1$-factor in each local trivialization.
		
		Now let $P\to M$ be the principal $S^1$-bundle with Euler class $e$. Since $M$ is closed and admits a Riemannian metric of positive Ricci curvature, its fundamental group is finite. Hence, by \cite[Proposition 6.2]{Su63}, see also \cite{HY76}, we can lift the $T^k$-action on $M$ to $P$ so that it commutes with the free $S^1$-action on $P$, hence we obtain a $T^{k+1}$-action on $P$.
		
		Next, as shown in \cite[p.\ 459]{GPT98}, there exists a connection form on $P$ with harmonic curvature form representing $-2\pi e\in H^2_{dR}(M)$ as in Proposition \ref{P:CURV_FORM} which is $T^k$-invariant. Hence, the metric $g_\phi$ for any $\phi>0$ is $T^{k+1}$-invariant.
		
		We now proceed along the same lines as in the proof of Theorem \ref{T:Twist_sus_Ric_multiple}. Since the action of $T^{k+1}$ on each boundary component of $P\setminus\bigcup_i\varphi_i(D^m\times S^1)^\circ$, which we identify with $S^{n-1}\times S^1$ via $\varphi_i$, is linear on each factor and constant in orthogonal direction to the boundary, the doubly warped submersion metric $\check{g}_{f,h}$ constructed in the proof of Theorem \ref{T:Twist_sus_Ric_multiple} is invariant under the action as long as we choose the connection form $A$ on $[0,s_\lambda]\times S^{n-1}\times S^1$ to be invariant.
		
		Finally, the deformation at the end of the proof of Theorem \ref{T:Twist_sus_Ric_multiple}, which ensures that the resulting metric has strictly positive Ricci curvature, can be chosen to preserve the isometry group as explained in \cite[p.\ 20]{Eh76}.
	\end{proof} 
	
	\begin{proof}[Proof of Theorem \ref{T:cohomo_2}]
		We start by considering the action of $S^1$ on $S^3\subset \C^2$ given by $\lambda(z_1,z_2)=(\lambda z_1,z_2)$. Since the action is linear, it preserves the round metric on $S^3$. Its fixed point set is given by $\{0\}\times S^1$, in particular, we have an infinite number of fixed points. By Theorem \ref{T:Twist_susp_Ric_equiv}, it follows that the manifold
		\[ \Sigma_0 S^3\#_{\ell}(S^2\times S^2)\cong \#_\ell(S^2\times S^2) \]
		admits a Riemannian metric of positive Ricci curvature that is invariant under an effective $T^2$-action for any $\ell\in\N_0$. This proves Theorem \ref{T:cohomo_2} for $n=4$ and $k$ even.
		
		For the remaining cases first note that the $T^2$-action on $\#_\ell (S^2\times S^2)$ has precisely $2\ell+2$ fixed points. This follows for example from the fact that the fixed point set has the same Euler characteristic as the manifold itself, see \cite{Ko58}. It can also seen directly from the construction, where the fixed point set on each attached $S^2\times D^2$ is given by $\{x_0,x_1\}\times\{0\}$, where $x_0,x_1$ are the two fixed points of the rotational action of $S^1$ on $S^2$.
		
		It then follows from Theorem \ref{T:Twist_susp_Ric_equiv}, that the manifold
		\[ M_{e,\ell'}=(\Sigma_e\#_\ell (S^2\times S^2 ))\#_{\ell'}(S^2\times S^3) \]
		for any $e\in H^2(\#_\ell (S^2\times S^2);\Z)$ and $\ell'\in\N_0$ admits a Riemannian metric of positive Ricci curvature that is invariant under a $T^3$-action. Further, by Lemma \ref{L:Twist_susp_topology}, the manifold $M_{e,\ell'}$ is spin if and only if $e$ has even divisibility and $M_{e,\ell'}$ has second Betti number given by $2\ell+\ell'$. In particular, we can realize any number $k\geq 0$ as the second Betti number of a spin manifold and any number $k\geq 2$ as the second Betti number of a non-spin manifold (since for $\ell=0$ the class $e\in H^2(S^4;\Z)$ vanishes). Finally, for $k=1$, by \cite{Pa04}, see also \cite{GS11}, we can construct a non-spin manifold with a $T^3$-action as a biquotient, and by \cite{ST04} it also admits an invariant Riemannian metric of positive Ricci curvature. 
		
		Now suppose $n>5$ and $k\in\N_0$. Let $B^5$ be a manifold constructed above with $b_2(B)=k+n-5$ such that $B$ is spin if and only if $k=0$. Consider a principal $T^{n-5}$-bundle $P\xrightarrow{\pi}B$ with simply-connected total space. This can be achieved by choosing an Euler class $e(\pi)\in H^2(B;\Z)^{n-5}$ that can be extended to a basis of $H^2(B;\Z)\cong\Z^{k+n-5}$, and in this case we have
		\[ b_2(P)=b_2(B)-(n-5)=k, \]
		see e.g.\ \cite[Lemma 2.3]{GR23}. Further, by \cite[Corollary 2.6]{GR23}, the manifold $P$ is spin if and only if $w_2(B)\in H^2(B;\Z/2)$ is contained in the subspace generated by $e(\pi)\mod 2$. Thus, by choosing a suitable Euler class, we can realize $P$ as either a spin or non-spin manifold whenever $k>0$.
		
		By \cite[Theorem 0.1]{GPT98}, the manifold $P$ admits a Riemannian metric of positive Ricci curvature that is invariant under the free $T^{n-5}$-action. Further, as explained in \cite[p.\ 459]{GPT98}, the $T^3$-action on $B$ constructed above can be lifted to $P$ such that it commutes with the $T^{n-5}$-action and the Riemannian metric can be chosen to be invariant under this action. Thus, we obtain an isometric $T^{n-2}$-action on the $n$-dimensional manifold $P$.
		
		Finally, assume that $M^n$ is a closed, 2-connected manifold with an effective action of $T^{n-2}$. By Proposition \ref{P:cohomo_b2}, after applying a torus automorphism, $M$ is equivariantly diffeomorphic to the manifold constructed above with $k=0$, thus admitting an invariant Riemannian metric of positive Ricci curvature.
	\end{proof}

	\begin{remark}\label{R:Ric>0_conn_sum}
		The manifolds constructed in the proof of Theorem \ref{T:cohomo_2} are total spaces of principal torus bundles $P\to M_{e,\ell'}$ over the manifolds $M_{e,\ell'}$. It follows from \cite[Theorems B and C]{GR23} that each total space $P$ is diffeomorphic to a connected sum of products of spheres and possibly a non-trivial linear sphere bundle over $S^2$ (in case $P$ is non-spin). The existence of Riemannian metrics of positive Ricci curvature (without any assumption on symmetries) on manifolds of this form has already been established in \cite{Bu20} in combination with \cite[Theorem C]{Re23}.
	\end{remark}
		
	\subsection{Connected sums and principal $S^1$-bundles}
	
	In this subsection we prove Corollary \ref{C:S1-bundle}, which is an immediate consequence of Corollary \ref{C:Twist_sus_conn_sum} together with the results of \cite{GR23}.
	
	\begin{proof}[Proof of Corollary \ref{C:S1-bundle}]
		For (1) let $M^{4m}$ be a closed manifold that admits a Riemannian metric of positive Ricci curvature, let $e\in H^2(M;\Z)$ and let $P\to M\#_\ell(\pm\C P^{2m})$ be the principal $S^1$-bundle with Euler class $e+\sum_{i=1}^\ell x_i$, where $x_i\in H^2(\C P^{2m};\Z)$ is a generator of the cohomology ring of the $i$-th $\pm\C P^{2m}$-summand.
		
		By \cite[Theorem A]{GR23}, we can decompose $P$ as follows:
		\[ P\cong\Sigma_e M\#\Sigma_{x_1}(\pm\C P^{2m})\#\dots\#\Sigma_{x_{\ell-1}}(\pm\C P^{2m})\# P_0,  \]
		where $P_0\to \C P^{2m}$ is the principal $S^1$-bundle with Euler class $x_\ell$.
		
		By \cite[Theorem B]{GR23}, each twisted suspension $\Sigma_{x_i}(\pm\C P^{2m})$ is diffeomorphic to $S^2\times S^{4m-1}$ and, since $x_\ell\in H^2(\C P^{2m})$ is a generator, we have $P_0\cong S^{4m+1}$. Hence $P$ is diffeomorphic to
		\[ P\cong \Sigma_e M\#_{\ell-1}(S^2\times S^{4m-1}) \]
		and the claim now follows from Corollary \ref{C:Twist_sus_conn_sum}.
		
		For (2) let $M^n$ be a closed, simply-connected manifold that admits a Riemannian metric of positive Ricci curvature and let $P\to M$ be a principal $S^1$-bundle with primitive Euler class $e$. Then, by \cite[Theorem B]{GR23}, we have
		\[ \Sigma_e M\cong P\# N. \]
		Hence, by Corollary \ref{C:Twist_sus_conn_sum}, the manifold
		\[ (P\# N)\#_{\ell-1}(S^2\times S^{n-1}) \]
		admits a Riemannian metric of positive Ricci curvature. The claim now follows from the fact that the manifold $N\#_{\ell-1}(S^2\times S^{n-1})$ is diffeomorphic to $\#_\ell N$, see e.g.\ \cite[Corollary 4.2]{GR23}.
	\end{proof}
	
	\begin{remark}
		Note that, in order to apply the results of \cite[Theorem B]{GR23}, we do not need that $M$ is simply-connected in (2) of Corollary \ref{C:S1-bundle}. Indeed, we only need to require that the fibre inclusion $S^1\hookrightarrow P$ is null-homotopic, or, equivalently, that the pull-back of the Euler class to the universal cover of $M$ is primitive. In this case we define $N$ as $S^2\times S^{n-1}$, whenever the universal cover of $M$ is non-spin and as the unique non-trivial linear $S^{n-1}$-bundle over $S^2$ otherwise.
	\end{remark}
	
	\subsection{Simply-connected 6-manifolds}
	
	In this subsection we prove Theorem \ref{T:Hom_Sphere}. The manifolds $\Sigma$ and $M_k$ in Theorem \ref{T:Hom_Sphere} will be constructed as twisted suspensions of corresponding 5-manifolds. Recall from Smale's classification of closed, simply-connected spin 5--manifolds \cite{Sm62} that for each $k\in\N$ there exists a unique closed, simply-connected spin 5-manifold $N_k$ with $H_2(N_k)\cong\Z/k\oplus\Z/k$.
	
	\begin{theorem}[{\cite[Corollary 10.2.20 and Table B.4.2]{BG08}, see also \cite{Ko05,Ko09}}]\label{T:pos_sasaki}
		The following manifolds admit a positive Sasakian structure and therefore a Riemannian metric of positive Ricci curvature:
		\begin{enumerate}
			\item The manifold $N_k$ whenever $k$ is not a multiple of $30$.
			\item The manifold $N_k\#_{7}(S^2\times S^3)$ for all $k\in\N$.
		\end{enumerate}
	\end{theorem}
	
	\begin{proof}[Proof of Theorem \ref{T:Hom_Sphere}]
		For (1) we define $\Sigma=\Sigma_0 N_k$ whenever $k$ is not a multiple of 30. By Theorem \ref{T:pos_sasaki} and Corollary \ref{C:Twist_sus_conn_sum} the manifold $\Sigma\#_\ell (S^2\times S^4)$ then admits a Riemannian metric of positive Ricci curvature for all $\ell\in\N_0$. 
		
		By Lemma \ref{L:Twist_sus_conn_sum}, the manifold $\Sigma$ is simply-connected. Further, it follows from Poincaré duality and the universal coefficient theorem that $H_3(N_k;\Z)$ and $H_4(N_k;\Z)$ are trivial. Hence, by \eqref{EQ:Twist_sus_hom}, the homology groups of $\Sigma$ are given by $H_2(\Sigma;\Z)\cong H_3(\Sigma;\Z)\cong H_2(N_k;\Z)\cong \Z/k\oplus \Z/k$ and $H_4(\Sigma;\Z)\cong H_3(N_k;\Z)\oplus H_4(N_k;\Z)=0$.
		
		For (2) we define $M_k=\Sigma_e(N_k\#_7(S^2\times S^3))$ for some
		\[e\in H^2(N_k\#_7(S^2\times S^3);\Z)\cong\bigoplus_7 H^2(S^2\times S^3;\Z).\]
		As before we obtain from Theorem \ref{T:pos_sasaki} and Corollary \ref{C:Twist_sus_conn_sum} that for any $\ell\in\N_0$ the manifold $M_k\#_{\ell}(S^2\times S^4)$ admits a Riemannian metric of positive Ricci curvature and by Lemma \ref{L:Twist_susp_topology} we have that $M_k$ is simply-connected with 
		\[H_2(M_k;\Z)\oplus H_2(N_k\#_7(S^2\times S^3);\Z)\cong \Bigslant{\Z}{k\Z}\oplus\Bigslant{\Z}{k\Z}\oplus \Z^7. \]
		Further, by Lemma \ref{L:Twist_susp_topology}, the manifold $M_k$ is spin if and only if $e$ has even divisibility, so that we can construct a spin manifold by choosing $e=0$ and a non-spin manifold by choosing $e$ to be primitive.
	\end{proof}
	
	\begin{remark}
		We can give further information on the topology of $\Sigma=\Sigma_0 N_k$ and $M_k=\Sigma_e (N_k\#_7(S^2\times S^3))$. First note that it follows from Lemma \ref{L:Twist_sus_conn_sum} that
		\[ M_k=\Sigma_e(N_k\#_7(S^2\times S^3))\cong\Sigma_0 N_k\#\Sigma_{e_1}(S^2\times S^3)\#\dots\#\Sigma_{e_7}(S^2\times S^3), \]
		where $e_i\in H^2(S^2\times S^3;\Z)$ denotes the restriction of $e$ to the $i$-th summand. Further, by \cite[Theorem B]{GR23}, each $\Sigma_{e_i}(S^2\times S^3)$ is diffeomorphic to either $(S^2\times S^4)\# (S^3\times S^3)$ or $(S^2\ttimes S^4)\# (S^3\times S^3)$, where $S^2\ttimes S^4$ denotes the total space of the unique non-trivial $S^4$-bundle over $S^2$, depending on whether $e_i$ has even or odd divisibility.
		
		It remains to investigate the topological properties of $\Sigma_0 N_k$. Since $H^4(N_k;\Z)$ is trivial and $N_k$ is spin, the characteristic classes $w_2(N_k)$ and $p_1(N_k)$ are trivial. Hence, by Lemma \ref{L:Twist_susp_topology}, also the classes $w_2(\Sigma_0 N_k)$ and $p_1(\Sigma_0 N_k)$ vanish.
		
		Further, for any coefficient ring, the product of any two classes in $H^2(\Sigma_0 N_k)$ vanishes. This can be seen by considering the space $M'=(S^2\ttimes S^3)\# N_k$, where $S^2\ttimes S^3$ denotes the total space of the unique non-trivial linear $S^3$-bundle over $S^2$. If $x^*\in H^2(S^2\ttimes S^3;\Z)\cong \Z$ denotes a generator, then the principal $S^1$-bundle with Euler class $x^*$ over $M'$ is diffeomorphic to $(S^3\times S^3)\#\Sigma_0 N_k$ by \cite[Theorem A]{GR23}. By the Gysin sequence, the bundle projection induces a surjective map
		\[ H^2(M')\to H^2((S^3\times S^3)\#\Sigma_0 N_k)\cong H^2(\Sigma_0 N_k). \]
		Since $H^4(M')$ is trivial, any product of elements of $H^2(M')$ vanishes, and therefore the same holds for $H^2(\Sigma_0 N_k)$.
		
		Finally, we note that closed, simply-connected spin $6$-manifolds were classified by Zhubr \cite{Zu75}. A consequence of this classification is that the diffeomorphism type of such a manifold $M$ is uniquely determined by the second homology group $H_2(M;\Z)$, the cohomology ring $H^*(M;\Z/n)$ for all $n\in\N\cup\{\infty\}$ (which can be identified with a class $\mu(M)\in H_6(K(H_2(M),2);\Z)$), the third Betti number $b_3(M)$ and a certain class $p(M)\in H^4(M;\Z)\cong H_2(M;\Z)$ with $4p(M)=p_1(M)$, up to automorphisms of $H_2(M;\Z)$. In our case, we have from the above that $\mu(\Sigma_0 N_k)$, $b_3(\Sigma_0 N_k)$ and $p_1(\Sigma_0 N_k)$ all vanish. We also have that $p(\Sigma_0 N_k)$ vanishes, which can be seen as follows:
		
		Let $W_k$ be the manifold obtained by plumbing as follows:
		\begin{center}
			\begin{tikzpicture}
				\begin{scope}[every node/.style={circle,draw,minimum height=2.5em}]
					\node[label=center:$\underline{D}^2_{N_k}$] (V1) at (1,0) {\phantom{0}};
					\node[label=center:$\underline{D}^5_{S^2}$] (V2) at (3,0) {\phantom{0}};
				\end{scope}
				\path[-](V1) edge["{$\scriptstyle +$}"] (V2);
			\end{tikzpicture}
		\end{center}
		As seen in Section \ref{S:PREL}, we have $\partial W_k= \Sigma_0 N_k$. Moreover, since $W_k$ is homotopy equivalent to $N_k\vee S^2$, we have an isomorphism
		\[H_2(W_k;\Z)\cong H_2(N_k;\Z)\oplus H_2(S^2;\Z).\]
		Since the inclusion $N_k\setminus D^5\hookrightarrow N_k$ induces an isomorphism on $H_2$, the inclusion $\Sigma_0 N_k\hookrightarrow W_k$ induces the map $x\mapsto (x,0)$ on $H_2$ according to this splitting. Hence, the class $[(\Sigma_0 N_k,\textrm{id}_{\Z/k\oplus\Z/k})]$ is trivial in the bordism group $\Omega_6^{\textrm{Spin}}(K(\Z/k\oplus \Z/k,2))$. As shown in \cite[3.13]{Zu75}, there is a homomorphism
		\[P\colon \Omega_6^{\textrm{Spin}}(K(\Z/k\oplus \Z/k,2))\to \Z/k\oplus \Z/k\]
		with $p(M)=P([(M,\textrm{id}_{\Z/k\oplus \Z/k})])$ for all $M$ with $H_2(M;\Z)\cong \Z/k\oplus \Z/k$, and therefore we have $p(\Sigma_0 N_k)=0$. Hence, the diffeomorphism type of $\Sigma_0 N_k$ is uniquely determined by the invariants $H_2(\Sigma_0 N_k;\Z)\cong \Z/k\oplus \Z/k$, $\mu(\Sigma_0 N_k)=0$, $b_3(\Sigma_0 N_k)=0$ and $p(\Sigma_0 N_k)=0$.
	\end{remark}
	
	\begin{remark}
		The manifolds in Theorem \ref{T:pos_sasaki} are the only known examples of closed, simply-connected rational homology $5$-spheres with a Riemannian metric of positive Ricci curvature, with the exception of the manifolds $\#_\ell N_2$ for all $\ell\in\N$, a finite number of sporadic examples, see \cite[Corollary 10.2.20]{BG08}, and the Wu manifold $W^5$. The latter is the homogeneous space $\mathrm{SU(3)}/\mathrm{SO(3)}$ and it follows from classical results of Nash \cite{Na79} that this space admits a Riemannian metric of positive Ricci curvature. The Wu manifold $W^5$ is non-spin with $H_2(W;\Z)\cong \Z/2$. We can apply a similar construction as in the proof of Theorem \ref{T:Hom_Sphere} to these manifolds, which provide further examples of rational homology $6$-spheres with a Riemannian metric of positive Ricci curvature.
	\end{remark}
	
	\begin{remark}
		We can iterate the construction in the proof of Theorem \ref{T:Hom_Sphere} by taking iterated twisted suspensions of $N_k$ in Theorem \ref{T:pos_sasaki}. In this way we obtain infinite families of rational homology spheres with a Riemannian metric of positive Ricci curvature in any dimension $n>5$.
	\end{remark}
	
	\subsection{Non-simply-connected rational homology spheres}
	
	In this subsection we prove Theorem \ref{T:Space_forms}.

	\begin{proof}[{Proof of Theorem \ref{T:Space_forms}}]
		For (1) let $L_k$ be a $(n-1)$-dimensional lens space with $\pi_1(L_k)\cong\Z/k$. Recall that the cohomology of $L_k$ is given as follows:
		\[ H^i(L_k;\Z)\cong\begin{cases}
			\Z,\quad &i=0,n-1,\\
			\Bigslant{\Z}{k\Z},\quad &2\leq i\leq n-2\text{ and }i\text{ even},\\
			0,\quad else.
		\end{cases} \]
		Further, the cup product $H^2(L_k,\Z)\otimes H^{2i}(L_k;\Z)\to H^{2i+2}(L_k;\Z)$ for $1\leq i\leq \frac{n-4}{2}$ is given by the multiplication in the group $\Z/k$.
		
		Let $e\in H^2(L_k;\Z)$ be a generator and define $M_k=\Sigma_e L_k$. Then, by Lemma \ref{L:Twist_susp_topology}, we have $\pi_1(M_k)\cong\pi_1(L_k)\cong\Z/k$. Further, it follows from the Gysin sequence and Lemma \ref{L:Twist_susp_topology} that $M_k$ has trivial cohomology groups in degrees $3\leq i\leq n-2$ and $H^{n-1}(M_k;\Z)\cong \Z/k$. By Poincaré duality and the universal coefficient theorem, it follows that $M_k$ has vanishing homology groups in degrees $2\leq i\leq n-1$.
		
		Finally, by Corollary \ref{C:Twist_sus_conn_sum}, for any $\ell\in\N_0$ the manifold
		\[ M_k\#_\ell(S^2\times S^{n-2}) \]
		admits a Riemannian metric of positive Ricci curvature.
		
		For (2) let $\Sigma^3$ be the Poincaré homology sphere and define $S_n=\Sigma_0\dots\Sigma_0 \Sigma^3$ as the $(n-3)$-fold twisted suspension of $\Sigma^3$. Since $\Sigma^3$ is a homology sphere, also $S_n$ is a homology sphere by \eqref{EQ:Twist_sus_hom} and we have $\pi_1(S_n)\cong\pi_1(\Sigma^3)$ by Lemma \ref{L:Twist_susp_topology}. Finally, since $\Sigma^3$ is a spherical space form, it follows from Corollary \ref{C:Twist_sus_conn_sum} that for any $\ell\in\N_0$ the manifold
		\[ S_n\#_\ell(S^2\times S^{n-2}) \]
		admits a Riemannian metric of positive Ricci curvature.
	\end{proof}
	\begin{remark}
		More generally, since any orientable spherical space form $S_G=S^n/G$ is a rational homology sphere, the twisted suspension $\Sigma_e S_G$ is a rational homology sphere and we have $\pi_1(\Sigma_e S_G)\cong\pi_1(S_G)\cong G$ by Lemma \ref{L:Twist_susp_topology}. The universal cover of this space can be described as follows:
		
		Let $P\to S_G$ be the principal $S^1$-bundle with Euler class $e\in H^2(S_G;\Z)$. If we pull back this bundle along the covering $S^n\to S_G$, we obtain a principal $S^1$-bundle over $S^n$, which is trivial as $H^2(S^n;\Z)=0$. For an embedding $D^n\hookrightarrow S_G$ we then have $|G|$ preimages in $S^n$ of this embedding. Hence, the space $\Sigma_e S_G$ is covered by the space obtained from $S^n\times S^1$ by performing $|G|$ surgeries along local trivializations. By Lemma \ref{L:Twist_susp_multiple} this space is diffeomorphic to
		\[ \Sigma_0 S^n\#_{|G|-1}(S^2\times S^{n-1}). \]
		Since $\Sigma_0 S^n\cong S^{n+1}$, cf.\ \cite[Example 5.3]{GR23}, we obtain that the universal cover of $\Sigma_e S_G$ is given by
		\[ \widetilde{\Sigma_e S_G}\cong\#_{|G|-1}(S^2\times S^{n-1}). \]	
	\end{remark}
	
	\subsection{Manifolds with prescribed third homology}
	
	In this subsection we prove Theorem \ref{T:prescr_G}. The manifold $M_G$ will be constructed as the connected sum of twisted suspensions of $\C P^m$.
	
	\begin{lemma}\label{L:Twist_susp_CPm}
		Let $e\in H^2(\C P^m;\Z)$ and denote by $k\in\N_0$ the divisibility of $e$. If $k\neq 0$, then the cohomology of $\Sigma_e\C P^m$ is given as follows:
		\[ H^i(\Sigma_e\C P^m;\Z)\cong\begin{cases}
			\Z,\quad & i=0,2,2m-1,2m+1,\\
			\Bigslant{\Z}{k\Z},\quad & 4\leq i\leq 2m-2,\text{ and }i\text{ even},\\
			0,\quad & \text{else.}
		\end{cases} \]
	\end{lemma}
	\begin{proof}
		Let $P\to \C P^m\setminus D^{2m}$ be the principal $S^1$-bundle with Euler class $e$. It then follows from the Gysin sequence that the cohomology of $P$ is given by
		\[ H^i(P;\Z)\cong\begin{cases}
			\Z,\quad & i=0,2m-1,\\
			\Bigslant{\Z}{k\Z},\quad & 2\leq i\leq 2m-2,\text{ and }i\text{ even},\\
			0,\quad & \text{else.}
		\end{cases} \]
		This can alternatively also be seen from the fact that $\C P^m\setminus D^{2m}$ is homotopy equivalent to $\C P^{m-1}$ and therefore $P$ is homotopy equivalent to a lens space of dimension $(2m-1)$ with fundamental group $\Z/k$. The claim now follows from Lemma \ref{L:Twist_susp_topology}.
	\end{proof}
	\begin{proof}[Proof of Theorem \ref{T:prescr_G}]
		Recall that the group $G$ is given by
		\[ G\cong\bigslant{\Z}{k_1\Z}\oplus\dots\oplus\bigslant{\Z}{k_{\ell_1}\Z}\oplus \Z^{\ell_2}. \]
		Let $e_i\in H^2(\C P^m;\Z)$ be a class of divisibility $k_i$. We define
		\[ M_G=\Sigma_{e_1}\C P^m\#\dots\#\Sigma_{e_{\ell_1}}\C P^m\#_{\ell_2}\Sigma_0(S^2\times S^{2m-2}). \]
		By \cite[Theorem B]{GR23} we have
		\[\Sigma_0(S^2\times S^{2m-2})\cong (S^2\times S^{2m-1})\#(S^3\times S^{2m-2}).  \]
		Hence, by Lemma \ref{L:Twist_susp_CPm} we have
		\[ H^i(M_G;\Z)\cong\begin{cases}
			\Z,\quad & i=0,2m+1,\\
			\Z^\ell,\quad & i=2,2m-1,\\
			\Z^{\ell_2},\quad & i=3,\\
			G,\quad &i=2m-2,\\
			\bigslant{\Z}{k_1\Z}\oplus\dots\oplus\bigslant{\Z}{k_{\ell_1}\Z} ,\quad & 4\leq i\leq 2m-4,\text{ and }i\text{ even},\\
			0,\quad & \text{else.}
		\end{cases} \]
		In particular, by Poincaré duality, we have $H_3(M_G;\Z)\cong G$ and $M_G$ has the rational cohomology ring of $\#_{\ell}(S^2\times S^{2m-1})\#_{\ell_2}(S^3\times S^{2m-2})$.
		
		Finally, by Lemma \ref{L:Twist_sus_conn_sum}, we have that $M_G$ is a twisted suspension of the manifold
		\[ \#_{\ell_1}\C P^m\#_{\ell_2}(S^2\times S^{2m-2}). \]
		By \cite[Theorem C]{Bu19} and \cite[Theorem C]{Re23} each summand admits a core metric as defined in \cite{Bu19} and by \cite[Theorem B]{Bu19} the connected sum of manifolds with core metrics admits a Riemannian metric of positive Ricci curvature. Hence, by Theorem \ref{T:Twist_sus_Ric}, the manifold $M_G$ admits a Riemannian metric of positive Ricci curvature.
	\end{proof}

	\appendix
	
	\section{Cohomogeneity-two torus actions}\label{A:cohom_2}
	
	In this section we recall basic facts on cohomogeneity-two torus actions on closed, simply-connected manifolds and determine the second Betti number from the orbit space data. For basic facts on torus actions of cohomogeneity 2 we refer to \cite{GK14}, \cite{KMP74}, \cite{Oh82}, \cite{Oh83}, \cite{OR70} and the references therein. For a closed, simply-connected manifold $M^n$, $n\geq4$, with an effective action of $T^{n-2}$ we will use the following facts:
	\begin{enumerate}
		\item The only possible isotropy groups are the trivial group, $T^1$ and $T^2$.
		\item The orbit space is an $m$-gon, where $m$ denotes the number of orbits with $T^2$-isotropy. The action over the interior is free, has $T^1$-isotropy over the interior of each edge and $T^2$-isotropy over each vertex.
		\item Let $a_i\in\Z^{n-2}$ denote a primitive vector generating the $T^1$-isotropy over the $i$-th edge. Then each pair $(a_i,a_{i+1})$ and $(a_m,a_1)$ can be extended to a basis of $\Z^{n-2}$ and all vectors $a_1,\dots,a_m$ together generate $\Z^{n-2}$.
		\item Every labelling of vectors on an $m$-gon satisfying the properties of the previous item is induced by an effective action of $T^{n-2}$ on some closed, simply-connected $n$-dimensional manifold. Two such manifolds are equivariantly diffeomorphic if and only if there is a diffeomorphism between their orbit spaces that preserves the labellings (up to sign).
	\end{enumerate}
	
	\begin{proposition}\label{P:cohomo_b2}
		Let $M^n$ be a closed, simply-connected manifold that admits an effective action of a torus $T^{n-2}$. Let $m$ denote the number of orbits with isotropy $T^2$. Then the second Betti number of $M$ is given by
		\[ b_2(M)=m-n+2. \]
		In particular, $m\geq n-2$ and if $m=n-2$, then there is precisely one equivariant diffeomorphism type of such manifolds, up to automorphisms of $T^{n-2}$. 
	\end{proposition}
	\begin{proof}
		Let $A\in\Z^{(n-2)\times m}$ denote the matrix
		\[ A=(a_1\mid\dots\mid a_m). \]
		We extend $A$ to a unimodular matrix as follows. Since $a_1,\dots,a_m$ generate $\Z^{n-2}$, it follows from the elementary divisor theorem that there exist unimodular matrices $S\in\Z^{(n-2)\times (n-2)}$, $T\in\Z^{m\times m}$ such that
		\[ S^{-1}\cdot A\cdot T=\begin{pmatrix}
			\textrm{id}_{n-2} & 0
		\end{pmatrix}.\]
		Let $S'\in\Z^{m\times m}$ be defined by
		\[ S'=\begin{pmatrix}
			S & 0\\
			0 & \mathrm{id}_{m-n+2}
		\end{pmatrix}.\]
		Then the unimodular matrix
		\[ A'=S'\cdot T^{-1}\in\Z^{m\times m} \]
		is of the form
		\[ A'=\begin{pmatrix}
			A\\
			*
		\end{pmatrix}. \]
		We denote the $i$-th column of $A'$ by $a_i'$ and consider the $m$-gon with labelling $a_i'$ on the $i$-th edge. Let $M'$ denote the closed, simply-connected $(m+2)$-dimensional manifold with an action of $T^m$ inducing this orbit space labelling.
		
		Since each pair $(a_i,a_{i+1})$ and $(a_m,a_1)$ can be extended to a basis of $\Z^{n-2}$, each triple $(a_i',a_{i+1}',e_j)$ and $(a_m',a_1',e_j)$ with $j> n-2$, where $(e_1,\dots,e_m)$ denotes the standard basis of $\Z^m$, can be extended to a basis of $\Z^m$. This implies that the circle subgroups generated by each $e_j$ with $j>n-2$ intersect each isotropy subgroup of $M'$ trivially. Hence, the subtorus $\{0\}\times T^{m-n+2}$ of $T^m$ acts freely on $M'$ with quotient $M$. It follows that
		\[ b_2(M')=b_2(M)-m+n-2, \]
		see e.g.\ \cite[Lemma 2.3]{GR23}.
		
		Furthermore, we have $b_2(M')=0$, since otherwise, again by \cite[Lemma 2.3]{GR23}, we can consider a principal $S^1$-bundle over $M'$ with simply-connected total space $P'$. By lifting the action from $M'$ to $P'$, see \cite[Proposition 6.2]{HY76}, we obtain an effective action of $T^{m+1}$ on $P'$ with a free sub-action of a circle. Since the action is free, dividing out this action does not change the number of vertices in the orbit space. Thus, the orbit space of $P'$ is an $m$-gon labelled by $m$-elements of $\Z^{m+1}$ that generate $\Z^{m+1}$, which is a contradiction.
		
		Finally, if $m=n-2$, then $(a_1,\dots,a_m)$ is a basis of $\Z^m$, so for any other choice of labelling $(b_1,\dots,b_m)$ there exists an automorphism of $T^m$ carrying one into the other.
	\end{proof}

	\bibliographystyle{plainurl}
	\bibliography{References}

\end{document}